\documentclass{amsart}

\usepackage{amssymb}
\usepackage[utf8]{inputenc}

\usepackage[dvips]{graphicx}

\usepackage{color}

\theoremstyle{plain}

\newtheorem{mainthm}{Theorem}

\newtheorem{theorem}{Theorem}[section]
\newtheorem{proposition}[theorem]{Proposition}

\newtheorem{lemma}[theorem]{Lemma}

\theoremstyle{definition}

\newtheorem*{remark}{Remark}

\newtheorem*{qs}{Questions}

\newtheorem*{problem}{Problem}

\newcommand{\Z}{\mathbb{Z}}
\newcommand{\N}{\mathbb{N}}
\newcommand{\R}{\mathbb{R}}
\newcommand{\eps}{\varepsilon}

\begin{document}

\begin{abstract}
We characterize product Anosov diffeomorphisms in terms of the two-sided limit shadowing property. It is proved that an Anosov diffeomorphism is a product Anosov diffeomorphism if and only if any lift to the universal covering has the unique two-sided limit shadowing property. Then we introduce two maps in a suitable Banach space such that fixed points of these maps are related with shadowing orbits on the universal covering.
\end{abstract}

\title[Anosov diffeomorphisms]{Product Anosov diffeomorphisms and the two-sided limit shadowing property}

\author[Bernardo Carvalho]{Bernardo Carvalho}

\date{\today}

\thanks{2010 \emph{Mathematics Subject Classification}: Primary 37D20; Secondary 37C20.}
\keywords{Anosov, Hyperbolicity, Transitivity.}
\thanks{This paper was partially supported by CNPq (Brazil).}

\address[Bernardo Carvalho]{Universidade Federal do Rio de Janeiro - UFRJ}
\email{bmcarvalho06@gmail.com}
\maketitle

\section{Introduction}

An \emph{Anosov diffeomorphism} on a smooth manifold $M$ is a smooth diffeomorphism $f:M\to M$ satisfying:
\begin{enumerate}
  \item for every $x\in M$ there is a splitting $T_xM=E^s(x)\oplus E^u(x)$ which is invariant under the derivative map $Df(x):T_xM\to T_{f(x)}M$, that is, $$Df(x)(E^s(x))=E^s(f(x)) \,\,\, \textrm{and} \,\,\, Df(x)(E^u(x))=E^u(f(x)).$$ $E^s(x)$ is called the \emph{stable space} of $x$ and $E^u(x)$ is called the \emph{unstable space} of $x$.
  \item there exist a Riemannian metric in $M$ and a constant $0<\lambda<1$ such that $$|Df^k(x)v|_{f^k(x)}\leq\lambda^k|v|_x \,\,\, \textrm{and} \,\,\, |Df^{-k}(x)w|_{f^{-k}(x)}\leq \lambda^k|w|_x$$ for all $v\in E^s(x)$, $w\in E^u(x)$, $k\in\Z$ and $x\in M$, where $|.|_x$ denotes the norm in $T_xM$ induced by the Riemannian metric. This metric is said to be \emph{adapted} to $f$.
\end{enumerate}

Such systems have been intensely studied since the work of S. Smale \cite{S} in 1966. He introduced several examples of Anosov diffeomorphisms and stated some questions about them that are not answered yet (to my best knowledge). The central problem on this theory is to understand all examples of Anosov diffeomorphisms (up to topological conjugacy). Smale conjectured that every Anosov diffeomorphism must be topologically conjugated to an Anosov automorphism on an infra-nilmanifold. The following properties are expected to be true:

\begin{enumerate}
  \item $M$ is an infra-nilmanifold and the universal covering is the Euclidean space,
  \item The lift of $f$ to $\R^n$ is topologically conjugated to a hyperbolic matrix,
  \item The stable and unstable foliations are on global product structure,
  \item $f$ admits a fixed point,
  \item If $M$ is compact and connected then $f$ is transitive.
\end{enumerate}

We recall some definitions. Denote by $d$ the distance in $M$ induced by the adapted metric and consider the sets $$W^s(x)=\{y\in M;d(f^k(y),f^k(x))\to0, k\to\infty\} \,\,\,\,\,\, \textrm{and}$$ $$W^u(x)=\{y\in M;d(f^k(y),f^k(x))\to0, k\to-\infty\}.$$ These sets are called the \emph{stable set} and the \emph{unstable set} of the point $x\in M$, respectively. For an Anosov diffeomorphism these sets are leafs of two respective foliations which we call \emph{stable foliation} and \emph{unstable foliation}.

The stable and unstable foliations are on \emph{global product structure} if for every $x,y\in M$ the leaves $\widetilde{W}^s(x)$ and $\widetilde{W}^u(y)$ intersect in exactly one point. Here $\widetilde{W}^s(x)$ and $\widetilde{W}^u(y)$ denote the lift of the stable and unstable leaves, respectively, to the universal covering. If this is the case we say that $f$ is a \emph{product Anosov diffeomorphism}. When $M$ is compact and connected the following three classes of diffeomorphisms are expected to be the same:

\begin{itemize}
  \item Anosov diffeomorphisms,
  \item Transitive Anosov diffeomorphisms,
  \item Product Anosov diffeomorphisms.
\end{itemize}

Product Anosov diffeomorphisms are transitive but the converse of this statement is not known. In \cite{C} transitive Anosov diffeomorphisms are characterized in terms of the \emph{two-sided limit shadowing property}. Recall that a homeomorphism $f$ on a metric space $(X,d)$ has the two-sided limit shadowing property if for every sequence $(x_k)_{k\in\Z}\subset X$ satisfying $$d(f(x_k),x_{k+1})\to0, \,\,\,\,\,\, |k|\to\infty$$ there exists $z\in X$ satisfying $$d(f^k(z),x_k)\to0, \,\,\,\,\,\, |k|\to\infty.$$ The sequence $(x_k)_{k\in\Z}$ is called a \emph{two-sided limit pseudo-orbit} and the point $z$ is said to \emph{two-sided limit shadows} $(x_k)_{k\in\Z}$ (see \cite{ACCO}, \cite{C} and \cite{CK} for more information on this property). The following is proved in \cite{C} but we state here since it is not stated as a theorem there:

\begin{theorem}
An Anosov diffeomorphism on a compact and connected manifold is transitive if and only if it has the two-sided limit shadowing property.
\end{theorem}

Our first result characterizes product Anosov diffeomorphisms in terms of the two-sided limit shadowing property. We say that $f$ has the \emph{unique two-sided limit shadowing property} when every two-sided limit pseudo-orbit is two-sided limit shadowed by a \emph{single point}.

\begin{mainthm}\label{utsls}
An Anosov diffeomorphism is a product Anosov diffeomorphism if and only if any lift to the universal covering has the unique two-sided limit shadowing property.
\end{mainthm}

So the following three classes of diffeomorphisms are expected to be the same:

\begin{itemize}
  \item Anosov diffeomorphisms,
  \item Anosov diffeomorphisms with the two-sided limit shadowing property,
  \item Anosov diffeomorphisms with the unique two-sided limit shadowing property on the universal covering.
\end{itemize}

We note that the uniqueness of the shadowing point and also the uniqueness of the intersection between the lifted stable and unstable leaves is a problem appart from the existence of these points. Though it is an interesting problem, on this note we will focus on the existence of these points. Thus the following are central questions:

\begin{qs}\label{questions}
Is it true that any lift of an Anosov diffeomorphism to the universal covering has the two-sided limit shadowing property? Is this true if we suppose that the Anosov diffeomorphism has the two-sided limit shadowing property?
\end{qs}

Trying to answer these questions we introduce some new techniques: for any two-sided limit pseudo-orbit $(x_k)_{k\in\Z}$ in the universal covering we discuss two maps $F$ and $G$ from a suitable Banach space to itself (see sections 3 and 4 for precise definitions) such that fixed points of these maps are related with points that two-sided limit shadows $(x_k)_{k\in\Z}$.

\begin{mainthm}\label{.}
There exists a bijection between the set of points that two-sided limit shadows $(x_k)_{k\in\Z}$ and the set of fixed points of $F$ and $G$.
\end{mainthm}

In section $2$ we prove Theorem \ref{utsls}, in section $3$ we introduce the map $F$ and in section $4$ we discuss the second map $G$.

\section{Lifting shadowing properties}

We denote by $M$ a closed and connected smooth $n$-dimensional manifold, by $\widetilde{M}$ its universal covering and by $\pi:\widetilde{M}\to M$ the covering projection. Any Riemannian metric in $M$ can be lifted to a Riemannian metric in $\widetilde{M}$ so as the covering map $\pi$ is a local isometry. Thus there exists $\eps_0>0$ such that for each $x\in\widetilde{M}$, $\pi$ maps the $\eps_0$-neighborhood of $x$ isometrically to the $\eps_0$-neighborhood of $\pi(x)$. We denote by $d$ the distance in $M$ induced by the Riemannian metric and by $\tilde{d}$ the distance in $\widetilde{M}$ induced by the lifted metric.

In \cite{H} (and also in \cite{KP}) it is proved that a homeomorphism $f:M\to M$ has the \emph{shadowing property} if and only if any lift of $f$ to the universal covering also has it. We prove that this holds for the \emph{limit shadowing property}. First we make all definitions. For any number $\delta>0$ we say that a sequence $(x_k)_{k\in\Z}\subset M$ is a \emph{$\delta$-pseudo-orbit} if $$d(f(x_k),x_{k+1})<\delta, \,\,\,\,\,\, k\in\Z.$$ This sequence is \emph{$\eps$-shadowed} by a point $z\in M$ if $$d(f^k(z),x_k)<\eps, \,\,\,\,\,\, k\in\Z.$$ We say that $f$ has the shadowing property if for every $\eps>0$ there exists $\delta>0$ such that every $\delta$-pseudo-orbit is $\eps$-shadowed. A sequence $(x_k)_{k\in\N}\subset M$ is a \emph{limit pseudo-orbit} if $$d(f(x_k),x_{k+1})\to0, \,\,\,\,\,\, k\to\infty.$$ This sequence is \emph{limit shadowed} if there exists $z\in M$ such that $$d(f^k(z),x_k)\to0, \,\,\,\,\,\, k\to\infty.$$ We say that $f$ has the limit shadowing property if every limit pseudo-orbit is limit shadowed. Information about these properties can be found in \cite{P}.

\begin{lemma}
If $f:M\to M$ is a homeomorphism and $\tilde{f}:\widetilde{M}\to\widetilde{M}$ is any lift of $f$ to the universal covering then $\tilde{f}$ has the limit shadowing property if and only if $f$ also has it.
\end{lemma}

\begin{proof}
Suppose that $\tilde{f}$ has the limit shadowing property and consider $(x_k)_{k\in\N}\subset M$ a limit pseudo-orbit for $f$. Choose $N\in\N$ such that $$d(f(x_k),x_{k+1})<\eps_0, \,\,\,\,\,\, k\geq N.$$ Note that from the definition of $\eps_0$ we have $$\eps_0<\min\{\tilde{d}(\tilde{x},\tilde{y}); x\in M, \tilde{x},\tilde{y}\in\pi^{-1}(x), \tilde{x}\neq\tilde{y}\}.$$ Thus for each choice of $y_N\in\pi^{-1}(x_N)$ there exists a unique limit pseudo-orbit $(y_k)_{k\geq N}$ of $\tilde{f}$ such that $y_k\in\pi^{-1}(x_k)$ and $d(f(y_k),y_{k+1})<\eps_0$ for every $k\geq N$. Since $\tilde{f}$ has limit shadowing there exists $z\in\widetilde{M}$ that limit shadows $(y_k)_{k\geq N}$. Therefore $\pi(\tilde{f}^{-N}(z))$ limit shadows $(x_k)_{k\in\N}$. Since this holds for every limit pseudo-orbit it follows that $f$ has the limit shadowing property.

Now suppose that $f$ has the limit shadowing property and consider $(x_k)_{k\in\N}\subset\widetilde{M}$ a limit pseudo-orbit for $\tilde{f}$. The sequence $(\pi(x_k))_{k\in\N}\subset M$ is a limit pseudo-orbit for $f$ and thus it is limit shadowed by $z\in M$. Choose $K\in\N$ such that $$d(f^k(z),\pi(x_k))<\eps_0, \,\,\,\,\,\, k\geq K.$$ There is a unique point $\tilde{z}\in\pi^{-1}(f^K(z))$ such that $\tilde{d}(\tilde{z},x_K)<\eps_0$. It is easy to check then that $\tilde{f}^{-K}(\tilde{z})$ limit shadows $(x_k)_{k\in\N}$. Since this holds for every limit pseudo-orbit it follows that $\tilde{f}$ has the limit shadowing property.
\end{proof}

\hspace{-0.4cm}\textbf{Proof of Theorem \ref{utsls}:} It is quite obvious that if the lift of $f$ to the universal covering has the unique two-sided limit shadowing property then $f$ is a product Anosov diffeomorphism. Then we only need to prove the other direction. Suppose that $f$ is a product Anosov diffeomorphism, $\tilde{f}$ is any lift of $f$ to the universal covering and $(x_k)_{k\in\Z}$ is any two-sided limit pseudo-orbit of $\tilde{f}$. Since $f$ is an Anosov diffeomorphism it has the limit shadowing property (see \cite{P} for a proof) and the previous lemma assures that $\tilde{f}$ also has it.  So there exist two points $z_1,z_2\in\widetilde{M}$ satisfying $$\tilde{d}(\tilde{f}^k(z_1),x_k)\to0, \,\,\,\,\,\, k\to-\infty \,\,\,\,\,\, \textrm{and} \,\,\,\,\,\, \tilde{d}(\tilde{f}^k(z_2),x_k)\to0, \,\,\,\,\,\, k\to\infty.$$ Since $f$ is a product Anosov diffeomorphism there exists a unique point $z\in W^u(z_1)\cap W^s(z_2)$. It is easy to see that $z$ two-sided limit shadows $(x_k)_{k\in\Z}$. Moreover, this point is unique because if a point two-sided limit shadows $(x_k)_{k\in\Z}$ then it belongs to $W^u(z_1)\cap W^s(z_2)$ which is a singleton. \qed

\begin{remark}
For an Anosov diffeomorphism $f:M\to M$ it is proved in \cite{C} that the two-sided limit shadowing property holds when the specification property is present. We note that the specification property does not lift to the universal covering though, it only makes sense on the compact scenario.
\end{remark}

\section{The general scenario}\label{B}

From now on we introduce some ideas on how to answer the questions made in introduction. Since we will work on the universal covering, and we do not want to make the notation too heavy, we will denote by $f_0:M\to M$ a homeomorphism and by $f:\widetilde{M}\to\widetilde{M}$ any lift of $f_0$ to the universal covering and not by $f$ and $\tilde{f}$ as we did in the previous section. We will also denote by $d$ the distance in $\widetilde{M}$ induced by the lifted metric. For every two-sided limit pseudo-orbit $(x_k)_{k\in\Z}\subset\widetilde{M}$ of $f$ we introduce a Banach space and a map $F$ from this space to itself such that fixed points of $F$ are in a bijective relation with the set of points that two-sided limit shadows $(x_k)_{k\in\Z}$. These ideas are based on the proof of Shadowing Lemma that can be found in \cite{P}.

We denote by $C$ the set of all bilateral sequences $\bar{v}=(v_k)_{k\in\Z}$ where $v_k\in T_{x_k}\widetilde{M}$ for each $k\in\Z$. Let $B$ denote the subset of $C$ consisting of bounded sequences, i.e., sequences $(v_k)_{k\in\Z}\in C$ that satisfy $$\sup_{k\in\Z}|v_k|_{x_k}<\infty,$$ where $|.|_{x_k}$ is the norm in $T_{x_k}\widetilde{M}$ induced by the lifted metric. The map $||.||:B\to\R^{+}$ defined by $$||\bar{v}||=\sup_{k\in\Z}|v_k|_{x_k}$$ is a norm in $B$ that makes $(B;||.||)$ a Banach space. We consider the subspace $C_0$ of $B$ as the space of sequences $(v_k)_{k\in\Z}\in B$ that satisfy $$|v_k|_{x_k}\to0, \,\,\, |k|\to\infty.$$ It is easy to see that $C_0$ is a closed subspace of $B$ with respect to the norm defined above, so it is also a Banach space. We define a map $F:C_0\to C_0$ as follows: for each sequence $\bar{v}=(v_k)_{k\in\Z}\in C_0$ we define $F(\bar{v})$ as the sequence $$F(\bar{v})_k=\exp^{-1}_{x_k}\circ f\circ\exp_{x_{k-1}}(v_{k-1}), \,\,\,\,\,\, k\in\Z.$$

\begin{remark}
If we suppose that the ambient manifold has non-positive sectional curvature then the lifted metric has no conjugated points and the exponential map $\exp_x:T_x\widetilde{M}\to\widetilde{M}$ is a global diffeomorphism for every $x\in\widetilde{M}$. In this case $F$ is well defined. This puts a restriction on the ambient manifold but it seems to be no problem to us since it is expected that the universal covering of a manifold supporting an Anosov diffeomorphism is the Euclidean space.
\end{remark}

\begin{lemma}
If $\bar{v}\in C_0$ then $F(\bar{v})\in C_0$.
\end{lemma}

\begin{proof}
A standard compactness argument (which we omit here) proves that $f$ is uniformly continuous on $\widetilde{M}$. Thus for each $\eps>0$ we can choose $0<\delta<\frac{\eps}{2}$ such that $d(x,y)<\delta$ implies $d(f(x),f(y))<\frac{\eps}{2}$. Since $\bar{v}\in C_0$ and $(x_k)_{k\in\Z}$ is a two-sided limit pseudo-orbit of $f$ we can choose $K\in\N$ such that for $|k|\geq K$ we have $$|v_k|_{x_k}<\delta \,\,\,\,\,\, \textrm{and} \,\,\,\,\,\, d(f(x_k),x_{k+1})<\delta.$$ Thus for $|k|>K$ we have $$d(\exp_{x_{k-1}}(v_{k-1}),x_{k-1})=|v_{k-1}|_{x_{k-1}}<\delta$$ which imply
\begin{eqnarray*}|F(\bar{v})_k|_{x_k}&=&|\exp^{-1}_{x_k}\circ f\circ\exp_{x_{k-1}}(v_{k-1})|_{x_k} \\
&=&d(f(\exp_{x_{k-1}}(v_{k-1})),x_k) \\
&\leq&d(f(\exp_{x_{k-1}}(v_{k-1})),f(x_{k-1}))+d(f(x_{k-1}),x_k) \\
&<&\frac{\eps}{2}+\frac{\eps}{2} \\
&=&\eps.
\end{eqnarray*}
This is enough to prove that $F(\bar{v})\in C_0$.
\end{proof}

\begin{theorem}\label{teorema homeo}
There exists a bijection between the set of fixed points of $F$ and the set of points that two-sided limit shadows $(x_k)_{k\in\Z}$.
\end{theorem}

\begin{proof}
For a two-sided limit pseudo-orbit $(x_k)_{k\in\Z}$ of $f$ suppose that $\bar{v}$ is a fixed point of $F$. Then the sequence $(\exp_{x_k}(v_k))_{k\in\Z}$ is an orbit that two-sided limit shadows $(x_k)_{k\in\Z}$. Indeed, for each $k\in\Z$ we have $$v_k=F(\bar{v})_k=\exp^{-1}_{x_k}\circ f\circ\exp_{x_{k-1}}(v_{k-1}),$$ which implies $$\exp_{x_k}(v_k)=f\circ\exp_{x_{k-1}}(v_{k-1}).$$ By induction we obtain $$\exp_{x_k}(v_k)=f^k(\exp_{x_0}(v_0)), \,\,\,\,\,\, k\in\Z.$$ Therefore, $$d(f^k(\exp_{x_0}(v_0)),x_k)=d(\exp_{x_k}(v_k),x_k)=|v_k|_{x_k}\to0, \,\,\,\,\,\, |k|\to\infty,$$ that is, $\exp_{x_0}(v_0)$ two-sided limit shadows $(x_k)_{k\in\Z}$.

Now suppose that $z$ two-sided limit shadows $(x_k)_{k\in\Z}$. For each $k\in\Z$ let $$v_k=\exp^{-1}_{x_k}(f^k(z)).$$ We have $\bar{v}=(v_k)_{k\in\Z}\in C_0$ since $$|v_k|_{x_k}=d(\exp_{x_k}(v_k),x_k)=d(f^k(z),x_k)\to0, \,\,\,\,\,\, |k|\to\infty.$$ Moreover, for each $k\in\Z$ $$F(\bar{v})_k=\exp^{-1}_{x_k}\circ f\circ\exp_{x_{k-1}}(v_{k-1})=\exp^{-1}_{x_k}(f^k(z))=v_k,$$ which proves $\bar{v}$ is a fixed point of $F$. These arguments construct the desired bijection.
\end{proof}

\begin{remark}
It is important to note that the Banach space $C_0$ and also the map $F$ depend on the two-sided limit pseudo-orbit $(x_k)_{k\in\Z}$.
\end{remark}

\section{The Anosov scenario}

The map $F$ can be defined for any homeomorphism $f_0:M\to M$, any lift $f:\widetilde{M}\to\widetilde{M}$ and any two-sided limit pseudo-orbit $(x_k)_{k\in\Z}$ of $f$. It is not expected that $F$ admit fixed points in all cases though but it is when $f_0$ is an Anosov diffeomorphism. On this section we use the hyperbolic structure of an Anosov diffeomorphism to define a new map $G$ in $C_0$ with more structure that $F$ but with the same fixed points.

We begin lifting the differentiable structure of $M$ to a differentiable structure of $\widetilde{M}$ such that $\pi$ is a local diffeomorphism. In particular if $\pi(x)=x_0$ then the derivative map $D\pi(x):T_x\widetilde{M}\to T_{x_0}M$ is a linear isomorphism. So the splitting $T_{x_0}M=E^s(x_0)\oplus E^u(x_0)$ can be lifted to a splitting $T_x\widetilde{M}=E^s(x)\oplus E^u(x)$ that is invariant by $Df(x)$. One can also lift the adapted metric in $M$ to a Riemannian metric in $\widetilde{M}$ so that $f$ is an Anosov diffeomorphism and that the lifted metric is adapted to $f$.

For each $k\in\Z$ consider a linear isomorphism $I_k:T_{f(x_{k-1})}\widetilde{M}\to T_{x_k}\widetilde{M}$ satisfying
\begin{enumerate}
\item $I_k(E^s(f(x_{k-1})))=E^s(x_k),$
\item $I_k(E^u(f(x_{k-1})))=E^u(x_k),$
\item $|I_k(v)|_{x_k}\leq|v|_{f(x_{k-1})}$.
\end{enumerate}
Define a map $T:C_0\to C_0$ by $$T(\bar{v})_k=I_k\circ Df(x_{k-1})(v_{k-1}), \,\,\,\,\,\, k\in\Z.$$ Let $Id$ denote the identity map in $C_0$.

\begin{theorem} \label{bolado}
The map $Id-T$ is a bounded linear isomorphism in $C_0$ with bounded inverse $(Id-T)^{-1}$.
\end{theorem}

Using this theorem we can define the map $G:C_0\to C_0$ by $$G(\bar{v})=(Id-T)^{-1}\circ(F-T)(\bar{v}).$$ By definition $G$ and $F$ have the same fixed points in $C_0$, so Theorem \ref{teorema homeo} is enough to prove Theorem \ref{.}.

The map $(Id-T)^{-1}$ will be defined in the proof of Theorem \ref{bolado}. The map $F-T$ is the following: $$(F-T)(\bar{v})_k=\exp^{-1}_{x_k}\circ f\circ\exp_{x_{k-1}}(v_{k-1})-I_k\circ Df(x_{k-1})(v_{k-1}), \,\,\,\,\,\, k\in\Z.$$ We do not know how to obtain fixed points for this map in the general case but we do when $f$ is linear or when the numbers $d(f(x_k),x_{k+1})$ are sufficiently small. In the first case $G$ is a constant map and in the second case $G$ is a contraction in an invariant small neighborhood of $\bar{0}$ in $C_0$ (see \cite{P}). We hope some fixed point Theorem applies in the general case.

\section{Proof of Theorem \ref{bolado}}

Now we turn our attention to the proof of Theorem \ref{bolado}. For each $k\in\Z$ we consider projections $\pi^s_k:T_{x_k}\widetilde{M}\to E^s(x_k)$ and $\pi^u_k:T_{x_k}\widetilde{M}\to E^u(x_k)$ parallel to $E^u(x_k)$ and $E^s(x_k)$ respectively. Since $M$ is compact we can choose $N\in\N$ such that for every $k\in\Z$ we have $$|\pi^s_k(v)|_{x_k}\leq N|v|_{x_k} \,\,\,\,\,\, \textrm{and} \,\,\,\,\,\, |\pi^u_k(v)|_{x_k}\leq N|v|_{x_k}.$$ For each $k\in\Z$ consider the map $A_k:T_{x_k}\widetilde{M}\to T_{x_{k+1}}\widetilde{M}$ defined by $$A_k(v)=I_{k+1}\circ Df(x_k)(v).$$ Since $A_k(E^s(x_k))=E^s(x_{k+1})$ for every $k\in\Z$ we can compose these maps to obtain $$A_{k-1}\circ\dots\circ A_n(E^s(x_n))=E^s(x_k), \,\,\,\,\,\, n<k.$$ Note also that $A_k^{-1}(E^u(x_{k+1}))=E^u(x_k)$ so we analogously have $$A_k^{-1}\circ\dots\circ A_n^{-1}(E^u(x_{n+1}))=E^u(x_k), \,\,\,\,\,\, n\geq k.$$ Thus we can define a map $\mathcal{G}:C_0\to C_0$ as follows: for each $\bar{v}=(v_k)_{k\in\Z}\in C_0$ the sequence $\mathcal{G}(\bar{v})=(\mathcal{G}(\bar{v})_k)_{k\in\Z}$ is defined by $$\mathcal{G}(\bar{v})_k=\pi^s_k(v_k)+\sum_{n=-\infty}^{k-1}A_{k-1}\circ\dots\circ A_n\pi^s_n(v_n)-\sum_{n=k}^{+\infty}A_k^{-1}\circ\dots\circ A_n^{-1}\circ\pi^u_{n+1}(v_{n+1}).$$ We will prove that $\mathcal{G}$ is the inverse of the map $Id-T$. First note that for every $v\in E^s(x_k)$, $w\in E^u(x_{k+1})$ and $k\in\Z$ we have $$|A_k(v)|_{x_{k+1}}\leq\lambda|v|_{x_k} \,\,\,\,\,\, \textrm{and} \,\,\,\,\,\, |A_k^{-1}(w)|_{x_k}\leq\lambda|w|_{x_{k+1}}.$$ Thus for every $v\in E^s(x_n)$ and $n<k$ we have $$|A_{k-1}\circ\dots\circ A_n(v)|_{x_k}\leq\lambda^{k-n}|v|_{x_n}.$$ Also for every $w\in E^u(x_{n+1})$ and $n\geq k$ we have $$|A_k^{-1}\circ\dots\circ A_n^{-1}(w)|_{x_k}\leq\lambda^{n-k+1}|w|_{x_{n+1}}.$$ Therefore \begin{eqnarray*}\left|\sum_{n=-\infty}^{k-1}A_{k-1}\circ\dots\circ A_n\pi^s_n(v_n)\right|_{x_k}&\leq&\sum_{n=-\infty}^{k-1}\lambda^{k-n}|\pi^s_n(v_n)|_{x_n}\\
&\leq& N\sum_{n=-\infty}^{k-1}\lambda^{k-n}|v_n|_{x_n}\\
&\leq& N||\bar{v}||\sum_{n=-\infty}^{k-1}\lambda^{k-n}\\
&\leq& N||\bar{v}||\frac{\lambda}{1-\lambda}\end{eqnarray*} and \begin{eqnarray*}\left|\sum_{n=k}^{+\infty}A_k^{-1}\circ\dots\circ A_n^{-1}\circ\pi^u_{n+1}(v_{n+1})\right|_{x_k}&\leq&\sum_{n=k}^{+\infty}\lambda^{n-k+1}|\pi^u_{n+1}(v_{n+1})|_{x_{n+1}}\\
&\leq& N\sum_{n=k}^{+\infty}\lambda^{n-k+1}|v_{n+1}|_{x_{n+1}}\\
&\leq& N||\bar{v}||\sum_{n=k}^{+\infty}\lambda^{n-k+1}\\
&\leq& N||\bar{v}||\frac{\lambda}{1-\lambda}.\end{eqnarray*}

Thus for every $k\in\Z$ we have $$|\mathcal{G}(\bar{v})_k|_{x_k}\leq N||\bar{v}||+N||\bar{v}||\frac{\lambda}{1-\lambda}+N||\bar{v}||\frac{\lambda}{1-\lambda}=N\left(\frac{1+\lambda}{1-\lambda}\right)||\bar{v}||.$$ This proves that $\mathcal{G}(\bar{v})_k\in T_{x_k}\widetilde{M}$ for every $k\in\Z$. We prove more:

\begin{proposition}\label{c0}
If $\bar{v}\in C_0$ then $\mathcal{G}(\bar{v})\in C_0$.
\end{proposition}

First we prove an auxiliary lemma:

\begin{lemma}\label{2}
If $\bar{v}=(v_k)_{k\in\Z}\in C_0$ then $\sum_{n=1}^{k-1}\lambda^{k-n}|v_n|_{x_n}\to0$ when $k\to\infty$.
\end{lemma}

\begin{proof}
For each $\eps>0$ choose $K\in\N$ such that $|v_n|_{x_n}<\frac{\eps(1-\lambda)}{2}$ for every $n\geq K$. For $k>K$ we write $$\sum_{n=1}^{k-1}\lambda^{k-n}|v_n|_{x_n}=\sum_{n=1}^{K}\lambda^{k-n}|v_n|_{x_n}+\sum_{n=K+1}^{k-1}\lambda^{k-n}|v_n|_{x_n}.$$ Note that $$\sum_{n=1}^K\lambda^{k-n}|v_n|_{x_n}=\lambda^k\sum_{n=1}^K\lambda^{-n}|v_n|_{x_n}\to0, \,\,\,\,\,\, k\to\infty.$$ So we can choose $k\geq K$ such that $$\sum_{n=1}^{K}\lambda^{k-n}|v_n|_{x_n}<\frac{\eps}{2}.$$ Moreover, $$\sum_{n=K+1}^{k-1}\lambda^{k-n}|v_n|_{x_n}\leq\frac{\eps(1-\lambda)}{2}\sum_{n=K+1}^{k-1}\lambda^{k-n}\leq\frac{\eps(1-\lambda)}{2}\frac{1}{1-\lambda}=\frac{\eps}{2}.$$
Thus for each $\eps>0$ there is $k\in\N$ such that $$\sum_{i=1}^{n-1}\lambda^{n-i}|v_i|_{x_i}<\eps, \,\,\, n\geq k.$$
This finishes the proof.
\end{proof}

\hspace{-0.4cm}\textbf{Proof of Proposition \ref{c0}:}
Let $\bar{v}=(v_k)_{k\in\Z}\in C_0$. To prove that $\mathcal{G}(\bar{v})\in C_0$ we need to show that $|\mathcal{G}(\bar{v})_k|_{x_k}\to0$ when $|k|\to\infty$. To prove this we consider separately the three terms in the expression of $\mathcal{G}(\bar{v})_k$. The first one satisfies $$|\pi^s_k(v_k)|_{x_k}\leq N|v_k|_{x_k}\to0, \,\,\,\,\,\, |k|\to\infty.$$ For the second term it is enough to prove that $$\sum_{n=-\infty}^{k-1}\lambda^{k-n}|v_n|_{x_n}\to0, \,\,\,\,\,\, |k|\to\infty.$$ For $k>1$ we write $$\sum_{n=-\infty}^{k-1}\lambda^{k-n}|v_n|_{x_n}=\sum_{n=-\infty}^0\lambda^{k-n}|v_n|_{x_n}+\sum_{n=1}^{k-1}\lambda^{k-n}|v_n|_{x_n}.$$ The previous lemma shows that the second sum goes to zero when $k\to\infty$. For the first sum we have \begin{eqnarray*}\sum_{n=-\infty}^0\lambda^{k-n}|v_n|_{x_n}&=&\lambda^k\sum_{n=-\infty}^0\lambda^{-n}|v_n|_{x_n}\\
&\leq&\lambda^k||\bar{v}||\sum_{n=-\infty}^0\lambda^{-n}\\
&=&\lambda^k||\bar{v}||\frac{1}{1-\lambda}\to0, \,\,\,\,\,\, k\to\infty.\end{eqnarray*} For $k<0$ we argument as follows: for each $\eps>0$ we choose $K\in\N$ such that $|v_n|_{x_n}<\frac{\eps(1-\lambda)}{\lambda}$ for each $n\leq-K$. Thus for $k\leq-K$ we have $$\sum_{n=-\infty}^{k-1}\lambda^{k-n}|v_n|_{x_n}\leq\frac{\eps(1-\lambda)}{\lambda}\sum_{n=-\infty}^{k-1}\lambda^{k-n}=\frac{\eps(1-\lambda)}{\lambda}\frac{\lambda}{1-\lambda}=\eps.$$

For the third term in $\mathcal{G}(\bar{v})_k$ we can use the same arguments so we leave the details to the reader. \qed

\begin{lemma}\label{imp}
For each $\bar{v}\in C_0$ and each $k\in\Z$ the following holds: $$A_k(\mathcal{G}(\bar{v})_k)=\mathcal{G}(\bar{v})_{k+1}-v_{k+1}.$$
\end{lemma}

\begin{proof}
For each $k\in\Z$ we have \begin{eqnarray*}A_k(\mathcal{G}(\bar{v})_k)=A_k\circ\pi^s_k(v_k)&+&\sum_{n=-\infty}^{k-1}A_k\circ A_{k-1}\circ\dots\circ A_n\pi^s_n(v_n)-\pi^u_{k+1}(v_{k+1})\\
&-&\sum_{n=k+1}^{+\infty}A_{k+1}^{-1}\circ\dots\circ\pi^u_{n+1}(v_{n+1}).\end{eqnarray*}
To obtain the desired equality just put $-\pi^u_{k+1}(v_{k+1})=\pi^s_{k+1}(v_{k+1}) - v_{k+1}$ in the last one.
\end{proof}

\hspace{-0.4cm}\textbf{Proof of Theorem \ref{bolado}:}
We first prove that $Id-T$ is surjective. Indeed, $\mathcal{G}$ is a right inverse for $Id-T$. For every $w\in C_0$ and $k\in\Z$ we have: $$(Id-T)(\mathcal{G}(w))_k=\mathcal{G}(w)_k-A_{k-1}(\mathcal{G}(w)_{k-1})=w_k,$$ where the last equality is ensured by Lemma \ref{imp}. To prove that $Id-T$ is injective let $v\in C_0$ be such that $(Id-T)(v)=0$, that is, $T(v)=v$. Then $A_{k-1}(v_{k-1})=v_k$ for every $k\in\Z$. By induction we obtain $$v_k=A_{k-1}\circ\dots\circ A_0(v_0), \,\,\,\,\,\, k>0$$ and $$v_k=A_k^{-1}\circ\dots\circ A_{-1}^{-1}(v_0), \,\,\,\,\,\, k<0.$$ Now note that $$|v_k|_{x_k}\to0, \,\,\,\,\,\, |k|\to\infty,$$ and that this implies $$\pi^s_0(v_0)=\pi^u_0(v_0)=0,$$ otherwise $v_k$ would converge to $\infty$. Therefore $v_0=0$, which implies $v_k=0$ for every $k\in\Z$ and we are done.   \qed

\begin{remark}
The map $\mathcal{G}$ is a linear isomorphism in $C_0$, that is the inverse of $Id-T$ and has norm $$||\mathcal{G}||=\sup_{||\bar{v}||=1}||\mathcal{G}(\bar{v})||=\sup_{||\bar{v}||=1}\sup_{k\in\Z}|\mathcal{G}(\bar{v})_k|_{x_k}\leq N\frac{1+\lambda}{1-\lambda}.$$
\end{remark}

\begin{remark}
If $f$ is a partially hyperbolic diffeomorphism then we can consider the same operator $\mathcal{G}$ as defined above. The difference is that Lemma \ref{imp} does not hold as it is written. Indeed, the following holds: for each $\bar{v}\in C_0$ and each $k\in\Z$ we have $$A_k(\mathcal{G}(\bar{v})_k)=\mathcal{G}(\bar{v})_{k+1}-v_{k+1}+\pi^c_{k+1}(v_{k+1}),$$ where $\pi^c_k$ is the projection in the central direction $E^c(x_k)$ parallel to $E^s(x_k)\oplus E^u(x_k)$. In this case we have $$(Id-T)(\mathcal{G}(w))_k=\mathcal{G}(w)_k-A_{k-1}(\mathcal{G}(w)_{k-1})=w_k -\pi^c_k(w_k)$$ and $\mathcal{G}$ is not the inverse of $Id-T$ anymore.
\end{remark}

\begin{problem}An interesting problem is to understand how these techniques translate to the theory of Anosov flows. Are there a Banach space and a map on this space such that an analogous of Theorem \ref{.} holds? There are some examples of Anosov flows that are not product Anosov flows (see \cite{B}, \cite{F}) so if we answer this question affirmatively this map should not admit fixed points. 
\end{problem}



\end{document}